\documentclass[11pt]{article}
\usepackage{amsfonts,amsmath,latexsym,color,epsfig}
\usepackage{amsfonts,amsmath,latexsym,color,epsfig}
\newtheorem{theorem}{Theorem}[section]
\newtheorem{lemma}{Lemma}[section]
\newtheorem{proposition}{Proposition}[section]
\newtheorem{conjecture}{Conjecture}[section]

\newenvironment{proof}
      {\medskip\noindent{\bf Proof:}\hspace{1mm}}
      {\hfill$\Box$\medskip}

\input{epsf}

\makeatletter
\def\Ddots{\mathinner{\mkern1mu\raise\p@
\vbox{\kern7\p@\hbox{.}}\mkern2mu
\raise4\p@\hbox{.}\mkern2mu\raise7\p@\hbox{.}\mkern1mu}}
\makeatother

\title{\vspace{-0.7cm}Two extensions of Ramsey's theorem}
\author{David Conlon\thanks{Mathematical Institute, Oxford OX1 3LB, United
Kingdom.
E-mail: {\tt
david.conlon@maths.ox.ac.uk}. Research supported by a Royal Society University Research Fellowship.} \and Jacob Fox\thanks{Department of Mathematics, MIT,
Cambridge, MA 02139-4307. Email: {\tt fox@math.mit.edu}. Research
supported by a Simons Fellowship and NSF grant DMS-1069197.} \and Benny Sudakov\thanks{Department of Mathematics,
UCLA,  Los Angeles, CA 90095. Email: {\tt bsudakov@math.ucla.edu}. 
Research supported in part by NSF grant DMS-1101185, NSF CAREER award DMS-0812005 and by USA-Israeli BSF grant.} \and \\
Accepted for publication in Duke Mathematical Journal}
\date{}
\begin{document}
\maketitle

\begin{abstract}

Ramsey's theorem, in the version of Erd\H{o}s and Szekeres, states that every $2$-coloring of the edges of the complete graph on $\{1, 2,\ldots,n\}$ contains a monochromatic clique of 
order $\frac{1}{2}\log n$. In this paper, we consider two well-studied extensions of Ramsey's theorem.

Improving a result of R\"odl, we show that there is a constant $c>0$ such that every $2$-coloring of the edges of the complete graph on $\{2, 3,...,n\}$ contains a monochromatic clique 
$S$ for which the sum of $1/\log i$ over all vertices $i \in S$ is at least $c\log\log\log n$. This is tight up to the constant factor $c$ and answers a question of 
Erd\H{o}s from 1981.

Motivated by a problem in model theory, V\"a\"an\"anen asked whether for every $k$ there is an $n$ such that the following holds. For every permutation $\pi$ of $1,\ldots,k-1$, every 
$2$-coloring of the edges of the complete graph on $\{1, 2, \ldots, n\}$ contains a monochromatic clique $a_1<\ldots<a_k$ with 
$$a_{\pi(1)+1}-a_{\pi(1)}>a_{\pi(2)+1}-a_{\pi(2)}>\ldots>a_{\pi(k-1)+1}-a_{\pi(k-1)}.$$ 
That is, not only do we want a monochromatic clique, but the differences between consecutive 
vertices must satisfy a prescribed order. Alon and, independently, Erd\H{o}s, Hajnal and Pach answered this question affirmatively. Alon further conjectured that the true growth rate 
should be exponential in $k$. We make progress towards this conjecture, obtaining an upper bound on $n$ which is exponential in a power of $k$. This improves a result of Shelah, who 
showed that $n$ is at most double-exponential in $k$. 
\end{abstract}

\section{Introduction}

Ramsey theory refers to a large body of deep results in mathematics whose underlying philosophy is captured succinctly by the statement that ``Every large system contains a large well-organized subsystem.'' This subject is currently one of the most active areas of research within combinatorics, overlapping substantially with number theory, geometry, analysis, logic and computer science (see the book \cite{GRS80} for details). The cornerstone of this area is Ramsey's theorem, which guarantees the existence of Ramsey numbers. 

The Ramsey number $r(k)$ is the minimum $n$ such that in every $2$-coloring of the edges of the complete graph $K_n$ there is a monochromatic $K_k$. Ramsey's theorem \cite{R30} states that $r(k)$ exists for all $k$. Classical results of Erd\H{o}s \cite{E47} and Erd\H{o}s and Szekeres \cite{ES35} give the quantitative bounds $2^{k/2} \leq r(k) \leq 2^{2k}$ for $k \geq 2$. Over the last sixty years, there have been several improvements on these bounds (see, for example, \cite{Co}). However, despite efforts by various researchers, the constant factors in the above exponents remain the same. 

Given these difficulties, it is natural that the field has stretched in different directions. One such direction is to try to strengthen Ramsey's theorem, asking that the monochromatic 
clique have some additional structure. This allows us to test the limits of current methods and may also lead to the development of new techniques which could be relevant to the 
original problem of estimating classical Ramsey numbers. Furthermore, for some applications such additional structure is needed. One famous example of a theorem which strengthens Ramsey's theorem is the Paris-Harrington theorem \cite{PH77}. In this paper, we consider two further strengthenings, 
both of which have already been studied in some detail.

\subsection{Ramsey's theorem with skewed vertex distribution}

In the early 1980s, Erd\H{o}s, interested in the distribution of monochromatic cliques in edge-colorings, considered the following variant of Ramsey's theorem. For a finite set $S$ of integers greater 
than one, define its weight $w(S)$ by $$w(S)=\sum_{s \in S} \frac{1}{\log s}.$$ For a red-blue edge-coloring $c$ of the edges of the complete graph on $[2,n]=\{2,\ldots,n\}$, let 
$f(c)$ be the maximum weight $w(S)$ over all sets $S \subset [2,n]$ which form a monochromatic clique in coloring $c$. For each integer $n \geq 2$, let $f(n)$ be the minimum of $f(c)$ 
over all red-blue edge-colorings $c$ of the edges of the complete graph on $\{2,\ldots,n\}$. Note that a simple application of $r(k) \leq 2^{2k}$ only gives $f(n) \geq \frac{\log n}{2} 
\cdot \frac{1}{\log n}=\frac{1}{2}$.

In his paper {\it `On the combinatorial problems I would most like to see solved'}, Erd\H{o}s \cite{E81} conjectured that $f(n)$ tends to infinity and, furthermore, asked for an 
accurate estimate of $f(n)$. Soon after, R\"odl \cite{R03} verified this conjecture, showing that $f(n)=\Omega(\frac{\log \log \log \log n}{\log \log \log \log \log n})$. In the other 
direction, by considering a uniform random coloring of the edges, one can easily obtain $f(n)=O(\log \log n)$. R\"odl \cite{R03} 
improved this upper bound further to $f(n)=O(\log \log \log n)$. Nevertheless, there was still an exponential gap between 
the bounds for $f(n)$.

We will now discuss R\"odl's coloring that gives the estimate $f(n) = O(\log \log \log n)$. Cover the interval $[2,n]$ by $t=\lceil \log \log n \rceil$ intervals, where the $i$th interval is $[2^{2^{i-1}},2^{2^i})$. We first describe the 
coloring of the edges within each of these $t$ intervals, and then the coloring of the edges between these intervals. Using that the Ramsey number $r(k) \geq 2^{k/2}$, we can
edge-color the complete graph in the $i$th  interval so that the maximum monochromatic clique in
this interval has order $2^{i+1}$. Also note that the logarithm of any element in the $i$th interval  is at least $2^{i-1}$. Therefore, 
the maximum weight of any monochromatic clique in this interval is at most $4$. It follows again from the lower bound on $r(k)$ that there is a red-blue edge-coloring 
of the complete graph on $t=\lceil \log \log n \rceil$ vertices whose largest 
monochromatic clique is of order $O(\log t)$. Color the edges of the complete bipartite graph between the $i$th and $j$th interval by the color of edge $(i,j)$ 
in this coloring. We get a red-blue edge-coloring of the complete graph on $[2,n]$ such that any monochromatic clique in this coloring has a non-empty intersection with at most
$O(\log t)$ intervals. Since, as we explained above, every interval can contribute  at most $4$ to the weight of this clique, the 
total weight of any monochromatic clique is $O(\log t)=O(\log \log \log n)$.

The key idea behind R\"odl's lower bound for $f(n)$ is to try and force the type of situation that arises in this upper bound construction. We follow the basic line of his argument but add two extra ideas, dependent random choice and a weighted variant of Ramsey's theorem, to achieve a tight result. That is, we prove that $f(n)=\Omega(\log \log \log n)$, which, by the above construction of R\"odl, is tight up to a constant factor. This determines the growth rate of 
$f(n)$ and answers Erd\H{o}s' question. 

\begin{theorem}\label{erdosrodl}
For $n$ sufficiently large, every $2$-coloring of the edges of the complete graph on the interval $\{2,\ldots,n\}$ contains a monochromatic clique with vertex set $S$ such that $$\sum_{s \in S} \frac{1}{\log s} \geq 2^{-8} \log \log \log n.$$ Hence, $f(n)=\Theta(\log \log \log n)$. 
\end{theorem}

Ramsey's theorem continues to hold if we use more than $2$ colors. We define the Ramsey number $r(k; q)$ to be the minimum $n$ such that in every $q$-coloring of the edges of the 
complete graph $K_n$ there is a monochromatic $K_k$. The upper bound proof of Erd\H{o}s and Szekeres \cite{ES35} implies that $r(k;q) \leq q^{qk}$. On the other hand, a simple product 
coloring shows that, for $q$ even, $r(k;q) \geq r(k;2)^{q/2} \geq 2^{kq/4}$. Phrased differently, we see that any $q$-coloring of $K_n$ contains a monochromatic clique of size $c_q \log 
n$ and that this is, up to the constant, best possible.

It therefore makes sense to consider the function $f_q(n)$, defined now as the minimum over all $q$-colorings of the edges of the complete graph on $\{2, 3, \dots, n\}$ of the maximum 
weight of a monochromatic clique. However, as observed by R\"odl, the analogue of Erd\H{o}s' conjecture for three colors instead of two does not hold. 
Indeed, again cover the interval $[2,n]$ by $t=\lceil \log \log n \rceil$ intervals, where the $i$th interval is $[2^{2^{i-1}},2^{2^i})$. The edges inside the intervals are colored
red-blue as in the above construction and the edges between the intervals are colored green. Then the maximum weight of any red or blue clique is at most $4$, since any such clique must lie completely within one of the intervals, and the maximum weight of the green clique is at most $\sum_{i\geq 1} 2^{-i+1} \leq 2$.

\subsection{Ramsey's theorem with fixed order type}

We also consider another extension of Ramsey's theorem. For a positive integer $n$, let $[n]=\{1,\ldots,n\}$. Motivated by an application in model theory, Jouko V\"a\"an\"anen asked \cite{NeVa, Sh} whether, for any positive integers $k$ and $q$ and any permutation $\pi$ of $[k-1]$, there is a positive integer 
$R$ such that for any $q$-coloring of the edges of the complete graph on vertex set $[R]$ there is a monochromatic $K_k$ with vertices $a_1<\ldots<a_k$ satisfying 
$$a_{\pi(1)+1}-a_{\pi(1)}>a_{\pi(2)+1}-a_{\pi(2)}>\ldots>a_{\pi(k-1)+1}-a_{\pi(k-1)}.$$ 
That is, we not only want a monochromatic $K_k$, but the differences between consecutive vertices must satisfy a prescribed order. The least such positive integer $R$ is denoted by $R_{\pi}(k;q)$, and we let $R(k;q)=\max_{\pi} R_{\pi}(k;q)$, i.e., $R(k;q)$ is the maximum of $R_{\pi}(k;q)$ over all permutations $\pi$ of $[k-1]$. 

V\"a\"an\"anen's question was popularized by Joel Spencer. It was positively answered by Noga Alon and, independently, by Erd\H{o}s, Hajnal, and Pach \cite{EHP}. Alon's proof (see \cite{NeVa}) uses the Gallai-Witt theorem and gives a weak bound on $R(k;q)$. The proof by Erd\H{o}s, Hajnal, and Pach uses a compactness argument and gives no bound on $R(k;q)$. Later, Alon, Shelah and Stacey all independently found proofs giving tower-type bounds for $R(k;q)$, though these were never published (see \cite{Sh}). 

A natural conjecture, made by Alon (see \cite{Sh}), is that $R(k;q)$ should grow exponentially in $k$. For monotone sequences, this was confirmed by Alon and Spencer. A breakthrough on 
this problem was obtained by Shelah \cite{Sh}, who proved the double-exponential upper bound $R(k;q) \leq 2^{(q(k+1)^3)^{qk}}$. Here, we make further progress, showing that, for 
fixed $q$, $R(k;q)$ grows as a single exponential in a power of $k$.

\begin{theorem}\label{shelahorder} 
For any positive integers $k$ and $q$ and any permutation $\pi$ of $[k-1]$, every $q$-coloring of the edges of the complete graph on vertex set $[R]$ with $R=2^{k^{20q}}$ contains a monochromatic $K_k$ with vertices $a_1<\ldots<a_k$ satisfying 
$$a_{\pi(1)+1}-a_{\pi(1)}>a_{\pi(2)+1}-a_{\pi(2)}>\ldots>a_{\pi(k-1)+1}-a_{\pi(k-1)}.$$ 
That is, $R(k;q) \leq 2^{k^{20q}}$.
\end{theorem}

Common to the proofs of both Theorems \ref{erdosrodl} and \ref{shelahorder} is a simple, yet powerful lemma whose proof, which we present in the next section, uses a probabilistic argument 
known as dependent random choice. Early versions of this technique were developed in the papers \cite{G98, KR01, Su03}. Several variants have since been discovered and applied to various problems in 
Ramsey theory and extremal graph theory (see the survey \cite{FS10} for more details).

{\bf Organization of the paper.}
We prove Theorem \ref{erdosrodl} in Section \ref{sect3} and Theorem \ref{shelahorder} in Section \ref{sect4}. In Section \ref{sect3}, we make use of a weighted variant of Ramsey's theorem, Lemma \ref{WeightedRamsey}, which may be of independent interest. In Section \ref{sect6}, we make several additional related remarks. These include discussing the asymptotic behavior of $f(n)$, considering what happens for other weight functions, showing that some natural variants of both problems have simple counterexamples, and presenting a simple coloring that gives a lower bound on Ramsey numbers for cliques with increasing consecutive differences. All logarithms are base $2$ unless otherwise indicated. For the sake of clarity of presentation, we systematically omit floor and ceiling signs whenever they are not crucial. We also do not make any serious attempt to optimize absolute constants in our statements and proofs. 

\section{Dependent Random Choice} 

The following lemma shows that every dense graph contains a large vertex subset $U$ such that every small subset $S \subset U$ has many common neighbors. For a vertex $v$ in a graph, let $N(v)$ denote the set of neighbors of $v$. For a set $T$ of vertices, let $N(T)$ denote the set of common neighbors of $T$.

\begin{lemma}\label{drc}
Suppose $p>0$ and $s$, $t$, $N_1$, $N_2$ are positive integers satisfying 
${N_1 \choose s}(m/N_2)^t \leq p^tN_1/2.$ If $G=(V_1,V_2,E)$ is a bipartite graph with $|V_i|=N_i$ for $i=1,2$ and at least $pN_1N_2$ edges, then $G$ has a vertex subset $U \subset V_1$ such that $|U| \geq p^tN_1/2$ and every $s$ vertices in $U$ have at least $m$ common neighbors. 
\end{lemma}
\begin{proof}
Consider a set $T$ of $t$ vertices in $V_2$ picked uniformly at random with repetition. Let $W=N(T)$ and $X$ denote the cardinality of $W$. We have 
$$\mathbb{E}[X]=\sum_{v \in V_1}\left(\frac{|N(v)|}{N_2}\right)^t=N_2^{-t}\sum_{v \in V_1} |N(v)|^t \geq N_1 N_2^{-t}\left(\frac{\sum_{v\in V_1} |N(v)|}{N_1}\right)^t \geq p^tN_1,$$
where the second to last inequality is by Jensen's inequality applied to the convex function $f(z)=z^t$. 

Let $Y$ be the random variable which counts the number of subsets $S \subset W$ of size $s$ with fewer than $m$ common neighbors. For a given $S \subset V_1$, the probability that it is a subset of $W$ equals $\left(|N(S)|/N_2\right)^t$. Since there are at most ${N_1 \choose s}$ such sets, it follows that 
$$\mathbb{E}[Y] \leq {N_1 \choose s}\left(\frac{m}{N_2}\right)^t.$$ 

By linearity of expectation, $$\mathbb{E}[X-Y]=\mathbb{E}[X]-\mathbb{E}[Y] \geq p^tN_1 - {N_1 \choose s}\left(\frac{m}{N_2}\right)^t \geq p^tN_1/2,$$ where the last inequality 
uses the assumption of the lemma. 
Hence, there is a choice of $T$ such that the corresponding set $W$ satisfies $X-Y \geq p^tN_1/2$. Delete one vertex from each subset $S$ of $W$ of size $s$ with fewer than $m$ common neighbors. We let $U$ be the remaining subset of $W$. We have $|U| \geq X-Y \geq p^tN_1/2$ and all subsets of size $s$ have at least $m$ common neighbors. 
\end{proof}

\section{Monochromatic cliques of large weight}\label{sect3}

The off-diagonal Ramsey number is the smallest natural number $n$ such that any red-blue edge-coloring of $K_n$ contains either a red copy of $K_s$ or a blue copy of $K_t$. The Erd\H{o}s-Szekeres bound for Ramsey numbers says that for any $s, t \geq 2$,
\[r(s,t) \leq \binom{s+t-2}{s-1}.\]
Note that this implies $r(s,t) \leq 2^{s+t}$ and hence that every $2$-coloring of $K_n$ contains a monochromatic clique of order $\frac{1}{2}\log n$. 
The following lemma is a further simple consequence of this formula. Note that here and throughout the rest of this section we will use the natural logarithm $\ln$ as well as the log base $2$.

\begin{lemma} \label{ErdSzek}
Suppose $0 < a \leq \frac{1}{4}$. Then, every $2$-coloring of the edges of $K_n$ contains either a red clique of order $a \ln n$ or a blue clique of order $e^{\frac{1}{4a}} \ln n$.
\end{lemma}

\begin{proof}
From the Erd\H{o}s-Szekeres bound, we have
\[r(s, t) \leq \binom{s+t}{s} \leq \left(\frac{e(s+t)}{s}\right)^s.\]
Applying this with $s = a \ln n$ and $t = a (e^{\frac{1}{a} - 1} - 1) \ln n$ tells us that, since
\[\left(\frac{e(a \ln n + a (e^{\frac{1}{a} - 1} - 1) \ln n)}{a \ln n}\right)^{a \ln n} = e^{\ln n} = n,\]
there is either a red clique of order $s$ or a blue clique of order $t$. For $0<a \leq \frac{1}{4}$, we have 
\[a(e^{\frac{1}{a} - 1} - 1) \geq e^{\frac{1}{4a}}.\]
The result follows.
\end{proof}

We would now like to prove a weighted version of Ramsey's theorem. The set-up is that each vertex $v$ is given two weights $r_v$ and $b_v$ which are balanced in a certain sense. We would then like to show that it is possible to find a red clique $K$ or a blue clique $L$ for which either the sum of $r_v$ over the vertices of $K$ or the sum of $b_v$ over the vertices of $L$ is large. 

\begin{lemma} \label{WeightedRamsey}
Suppose that the edges of $K_n$ have been two-colored in red and blue and that each vertex $v$ has been given positive weights $r_v$ and $b_v$ satisfying $b_v \geq \ln (4/r_v)$ if $r_v \leq b_v$ and $r_v \geq \ln (4/b_v)$ if $b_v \leq r_v$. Then there exists either a red clique $K$ for which $\sum_{v \in K} r_v \geq \frac{1}{2} \ln n$ or a blue clique $L$ for which $\sum_{v \in L} b_v \geq \frac{1}{2} \ln n$.
\end{lemma}
\begin{proof}
Let $w(n)$ be the infimum, over all red-blue edge-colorings of $K_n$, for the sum of the maximum of $\sum_{v \in K}r_v$ over all red cliques $K$ and the maximum of $\sum_{v \in 
L}b_v$ over all blue cliques $L$. We will show by induction on $n$ that $w(n) \geq \ln n$. This clearly implies the desired bound. 

The base cases $n=1,2$ clearly hold. Suppose, therefore, that $n \geq 3$ and that $w(n') \geq \ln n'$ for all positive integers $n' < n$. 

Consider a red-blue edge-coloring of $K_n$, and let $w$ be the sum of the maximum of $\sum_{v \in K}r_v$ over all red cliques $K$ and the maximum of $\sum_{v \in L}b_v$ over all blue cliques $L$. It suffices to show that $w \geq \ln n$. 

Let $v$ be a vertex in $K_n$. By symmetry, we may suppose without loss of generality that $r_v \geq b_v$. Since $r_v \geq \ln (4/b_v)$ and $r_v \geq b_v$, we have $r_v \geq 1$.  We may assume $r_v \leq \ln n$ as otherwise we could pick the red clique $K$ to consist of just the vertex $v$. Hence, $b_v \geq 4/n$.

Let $R$ be the set of red neighbors of $v$ and $B$ be the set of blue neighbors of $v$, so $|R|+|B|=n-1$. Let $\alpha=|R|/n$. We can add $v$ to the largest red clique in $R$ in terms of weight, and thus $w \geq r_v+w(\alpha n) \geq r_v+\ln(\alpha n) \geq r_v+\ln \alpha+\ln n$. We may assume $r_v+\ln \alpha<0$, as otherwise we are done. So $\alpha<e^{-r_v} \leq \frac{b_v}{4}$. From $r_v \geq 1$, we have $\alpha < 1/e$. From the above lower bounds on $b_v$, we have $b_v \geq 4 \max(\alpha,\frac{1}{n}) \geq 2\beta$, where $\beta=\alpha+\frac{1}{n}<\frac{1}{e}+\frac{1}{3}<3/4$. We can add $v$ to the largest blue clique in $B$ in terms of weight, and thus 
$$w \geq b_v+w(|B|) \geq b_v+\ln \left(1-\alpha-\frac{1}{n}\right) + \ln n \geq 2\beta+\ln \left(1-\beta\right) + \ln n \geq \ln n,$$ 
where we used $0<\beta<3/4$, which completes the proof. 
\end{proof}

Scaling all weights by a factor $c>0$, we have the following equivalent version. 

\begin{lemma} \label{WeightedRamsey2} Let $c>0$. 
Suppose that the edges of $K_n$ have been two-colored in red and blue and that each vertex $v$ has been given positive weights $r_v$ and $b_v$ satisfying $b_v \geq c\ln (4c/r_v)$ if $r_v \leq b_v$ and $r_v \geq c\ln (4c/b_v)$ if $b_v \leq r_v$. Then there exists either a red clique $K$ for which $\sum_{v \in K} r_v \geq \frac{c}{2} \ln n$ or a blue clique $L$ for which $\sum_{v \in L} b_v \geq \frac{c}{2} \ln n$.
\end{lemma}

Before proving Theorem \ref{erdosrodl}, we sketch the proof. We begin with a collection of $d = O(\sqrt{\log \log n})$ disjoint subsets $S_{1,0}, S_{2,0}, \dots, S_{d,0}$ of $[2,n]$. We then find a sequence $T_{1,2} \supseteq T_{1,3} \supseteq \dots \supseteq T_{1,d}$ of subsets of $S_{1,0}$, where $T_{1,k}$ has the property that there exists a color $\chi(1, k)$, either red or blue, such that all small subsets of $T_{1,k}$ have many common neighbors in $S_{k,0}$ in that color. 

We now find maximum sized red and blue cliques in $T_{1,d}$. If either of these cliques is sufficiently large then that clique will be the desired clique for the whole theorem. We may therefore assume that they are both small. We let $S_{1,1} \subseteq T_{1,d} \subseteq S_{1,0}$ be their union. The fundamental property of $T_{1,d}$ now implies that, for $2 \leq k \leq d$, the set $S_{1,1}$ has many common neighbors in color $\chi(1,k)$ in $S_{k,0}$. We call this set of neighbors $S_{k,1}$. 

We now repeat the entire process, first within $S_{2,1}$ to find sets $S_{k, 2}$ for $2 \leq k \leq d$, then in $S_{3,2}$ to find sets $S_{k,3}$ for $3 \leq k \leq d$, and so on. In the end, we find subsets $S_{i,i}$ such that each $S_{i,i}$ is the union of a red and blue clique and the color of the edges between $S_{i,i}$ and $S_{j,j}$ only depends on $i$ and $j$. If we consider the reduced graph on $d$ vertices whose edges have color $\chi(i,j)$ and whose vertices have blue and red weights representing the sizes of the red and blue cliques in $S_{i,i}$, an application of Lemma \ref{WeightedRamsey2} allows us to complete the proof.

We are now ready to prove Theorem \ref{erdosrodl}, which we restate for convenience.

\begin{theorem} \label{RodlMain}
For sufficiently large $n$, in every red-blue edge-coloring of the complete graph on the interval $\{2, \dots, n\}$ there is a monochromatic clique with vertex set $I$ such that
\[\sum_{i \in I} \frac{1}{\log i} \geq 2^{-8}\log \log \log n.\]
\end{theorem}

\begin{proof}
Let $d=\frac{1}{2}\sqrt{\log \log n}-1$ and $c=1/4$. For $i = 1, \dots, d$, 
let $S_i=\{n_i, n_i+1, \ldots, 2n_i-1\}$ be the interval of size $n_i$ beginning at the integer $n_i$, 
where $\log \log n_i = i \sqrt{\log \log n} + \frac{1}{2} \log
\log n$. For each $j = 0, 1, \dots, d$, we will find, by induction, a collection of sets $S_{i,j}$ for $i \geq j$ such 
that $S_{i,j} \subseteq S_{i,j-1} \subseteq \dots \subseteq S_{i,0} = S_i$ and 

\begin{itemize}

\item
the set $S_{j,j}$ is the union of two monochromatic cliques, one in red of order $\frac{1}{4} r_j \log n_j$ and the other in blue of order $\frac{1}{4} b_j \log n_j$, where $r_j \geq c\ln (4c/b_j)$ if $b_j \leq r_j$ and $b_j \geq c\ln(4c/r_j)$ if $r_j \leq b_j$;

\item
for each $i > j$, the set $S_{i,j}$ satisfies $|S_{i,j}| \geq n_i^{1 - \frac{j}{2i}}$;

\item
for each $i > j$, there exists a color $\chi(j, i)$ such that every vertex in $S_{j, j}$ is connected to every vertex in $S_{i, j}$ by an edge with color $\chi(j, i)$.

\end{itemize} 

To begin the induction, we let $S_{i,0} = S_i$ for each $i$. The required conclusion then holds trivially for $j=0$. Suppose therefore that the result holds for $j$. We will prove it also holds for $j+1$. 

For each $i \geq j+1$, we will find a subset $S_{i,j+1}$ of $S_{i,j}$ satisfying the conditions. To do this we apply another induction, finding for each $j+1 \leq k \leq d$, a subset $T_{j+1, k}$ of $S_{j+1,j}$ such that $T_{j+1, d} \subseteq T_{j+1, d-1} \subseteq \dots \subseteq T_{j+1, j+1} = S_{j+1,j}$ and

\begin{itemize}
\item
$|T_{j+1, k}| \geq n_{j+1}^{\frac{1}{2} - \frac{k}{4d}}$;

\item
for every $j+1 < i \leq k$, there is a color $\chi(j+1, i)$ such that every collection of $\log n_{j+1} \log \log \log n$ vertices in $T_{j+1, k}$ have at least $n_i^{1 - 
\frac{j+1}{2i}}$ common neighbors in color $\chi(j+1, i)$ in the set $S_{i,j}$.

\end{itemize}

Once this induction is complete, we consider $T_{j+1,d}$. Let $\frac{1}{4}r_{j+1}\ln n_{j+1}$ and $\frac{1}{4}b_{j+1}\ln n_{j+1}$ denote the orders of the largest red clique and the largest blue clique, respectively, in $T_{j+1,d}$. Since $|T_{j+1,d}| \geq n_{j+1}^{1/4}$, the remark before Lemma \ref{ErdSzek} implies that if $r_{j+1} \geq b_{j+1}$ then $T_{j+1,d}$ contains a red clique of order $\frac{1}{2} \log |T_{j+1,d}| \geq \frac{1}{8} \log n_{j+1}$. That is, $r_{j+1} \geq \frac{1}{2}$. Using Lemma \ref{ErdSzek}, we see that if $r_{j+1} \geq b_{j+1}$ then either $b_{j+1} \geq 1/4$ and $r_{j+1} \geq 1/2 \geq c\ln \left(\frac{4c}{b_{j+1}}\right)$ or $b_{j+1} < 1/4$ and $r_{j+1} \geq e^{1/(4b_{j+1})} \geq c\ln\left(\frac{4c}{b_{j+1}}\right)$. Similarly, if 
$b_{j+1} \geq r_{j+1}$, then $b_{j+1} \geq c\ln \left(\frac{4c}{r_{j+1}}\right)$. 

Note that we may assume that $r_{j+1}$ and $b_{j+1}$ are each less than $\frac{1}{2} \log \log \log n$. Suppose otherwise and that $r_{j+1} \geq \frac{1}{2} \log \log \log n$. Let $R_{j+1}$ be the red clique of order $\frac{1}{4}r_{j+1}\log n_{j+1}$. Then
\[\sum_{i \in R_{j+1}} \frac{1}{\log i} \geq \frac{1}{4} r_{j+1} \ln n_{j+1} \frac{1}{\log 2 n_{j+1}} \geq \frac{1}{16} \log \log \log n,\]
so we would be done.

Let $S_{j+1, j+1}$ be the union of the largest red and blue cliques in $T_{j+1,d}$. Note that $r_{j+1} + b_{j+1} \leq \log \log \log n$. 
Hence, $|S_{j+1, j+1}| \leq \log n_{j+1} \log \log \log n$ and therefore, for every $j+1 < i \leq d$, the collection 
of vertices in $S_{j+1, j+1}$ has at least $n_i^{1 - \frac{j+1}{2i}}$ common neighbors, in color $\chi(j+1, i)$, in $S_{i, j}$. We let this set of common neighbors be $S_{i, j+1}$. It 
is now elementary to verify that the $S_{i,j+1}$ satisfy the conditions of the first induction.  Hence, it only remains to show that the second induction holds good.

To begin the induction, we let $T_{j+1, j+1}$ be $S_{j+1,j}$. This clearly satisfies the required conditions. Suppose, therefore, that $T_{j+1,k}$ has been defined and we now wish to find a subset $T_{j+1, k+1}$ of $T_{j+1, k}$ satisfying the conditions of the induction. Consider the graph between $T_{j+1, k}$ and $S_{k+1, j}$. Either red or blue will have density at least $\frac{1}{2}$ in this graph. We let $\chi(j+1, k+1)$ be such a color, breaking a tie arbitrarily. 

Now apply Lemma \ref{drc} to the bipartite graph of color $\chi(j+1,k+1)$ between $T_{j+1, k}$ and $S_{k+1,j}$. We take $N_1 = |T_{j+1,k}|$, $N_2 = |S_{k+1,j}|$, $m = N_2^{1 - 
1/2\sqrt{\log \log n}}$, $s = \log n_{j+1} \log \log \log n$ and $t = \frac{\log n_{j+1}}{ 4 \sqrt{\log \log n}}$. We need to verify that $\binom{N_1}{s} (\frac{m}{N_2})^t \leq 
\frac{p^t N_1}{2}$ with $p=1/2$. It will be enough to show that $N_1^s (\frac{m}{p N_2})^t \leq 1$. But this is easy to check, since \[N_1^s \left(\frac{m}{p N_2}\right)^t \leq 
(2N_1)^s N_2^{-t/2 \sqrt{\log \log n}} = \left(\frac{(2 N_1)^{\log \log \log n}}{N_2^{1/8 \log \log n}}\right)^{\log n_{j+1}} \leq \left(\frac{(2 n_{j+1})^{\log \log \log 
n}}{n_{k+1}^{1/16 \log \log n}}\right)^{\log n_{j+1}} < 1.\] Here we used that $N_1 = |T_{j+1,k}| \leq n_{j+1}$, $N_2 = |S_{k+1,j}| \geq n_{k+1}^{1/2}$ and, whenever $k > j$ and $n$ is 
sufficiently large, \[n_{k+1} \geq n_{j+1}^{2^{\sqrt{\log \log n}}} \geq (2 n_{j+1})^{16 \log \log n \log \log \log n}.\] Therefore, there exists a subset $M_{k+1}$ of $T_{j+1,k}$ of 
order $\frac{p^t N_1}{2}$ such that every vertex subset of order $s$ has at least $m$ common neighbors in $S_{k+1, j}$. We let $T_{j+1, k+1} = M_{k+1}$. 
Note that 
\[|T_{j+1,k+1}| \geq \frac{p^t |T_{j+1,k}|}{2} =\frac{1}{2}\, 2^{-\frac{\log n_{j+1}}{4 \sqrt{\log \log n}}} n_{j+1}^{\frac{1}{2} - \frac{k}{4d}}=
\frac{1}{2}\,n_{j+1}^{-\frac{1}{8(d+1)}} n_{j+1}^{\frac{1}{2} - \frac{k}{4d}} \geq 
n_{j+1}^{\frac{1}{2} - \frac{k+1}{4d}},\] 
as required. Moreover, since $k \leq d \leq  \frac{1}{2} \sqrt{\log \log n}$, every subset of $T_{j+1,k+1}$ 
of order $\log n_{j+1} \log \log \log n$ has at least \[m \geq |S_{k+1, j}|^{1 - 1/2 \sqrt{\log \log n}} \geq \left(n_{k+1}^{1 - \frac{j}{2(k+1)}}\right)^{1 - 1/2 \sqrt{\log \log n}} 
\geq n_{k+1}^{1 - \frac{j+1}{2(k+1)}}\] common neighbors in $S_{k+1,j}$, so the second requirement of the induction scheme also holds.

To complete the proof, note that for each $i = 1, \dots, d$, we have found a red clique $R_i$ and a blue clique $B_i$ of orders $\frac{1}{4}r_i \ln n_i$ and $\frac{1}{4}b_i \ln n_i$, respectively, such that every vertex in $R_i \cup B_i$ is connected to every vertex in $R_j \cup B_j$ by color $\chi(i, j)$. Consider the $2$-colored complete graph on the vertex set $\{1, 2, \dots, d\}$ where $i$ and $j$ are joined in color $\chi(i,j)$. We give each vertex the two weights $r_i$ and $b_i$. Since $b_i \geq c\ln (4c/r_i)$ if $r_i \leq b_i$ and $r_i \geq c\ln (4c/b_i)$ if $b_i \leq r_i$, we may apply Lemma \ref{WeightedRamsey2} to find a red clique $R$ such that 
\[\sum_{i \in R} r_i \geq \frac{c}{2} \ln d \geq \frac{1}{32} \log \log \log n\] 
or a blue clique $B$ such that $\sum_{i \in B} b_i \geq \frac{1}{32} \log \log \log n$. Suppose, without loss of generality, that there is a red clique $R$ such that $\sum_{i \in R} r_i \geq \frac{1}{32} \log \log \log n$. 

Consider now the set $\mathcal{R} = \bigcup_{i \in R} R_i$. Since $R$ is a red clique by coloring $\chi$, the edges between different $R_i$ are red. Therefore, since also each $R_i$ is a red clique, we see that $\mathcal{R}$ is a red clique in the original graph. Moreover,
\[\sum_{j \in \mathcal{R}} \frac{1}{\log j} \geq \sum_{i \in R} \sum_{j \in R_i} \frac{1}{\log j} \geq \sum_{i \in R} \frac{r_i}{4} \log n_i \frac{1}{\log 2 n_i} \geq \sum_{i \in R} \frac{r_i}{8} \geq 2^{-8} \log \log \log n,\]
as required.
\end{proof}

\section{Monochromatic sets with differences satisfying a prescribed order}\label{sect4}

In this section we prove Theorem \ref{shelahorder}, which gives an improved bound for Ramsey numbers with fixed order type. We begin with several simple definitions and lemmas. 

An {\it interval} $I$ of integers is a set of consecutive integers. Let $S$ be a nonempty set of integers, and  $\min(S)$ and $\max(S)$ denote the minimum and maximum integers in $S$. The density $d_I(S)$ of $S$ with respect to an interval $I$ of integers with $S \subset I$ is $|S|/|I|$. 

The following definition is useful for finding cliques of a certain order type.  

\begin{definition}
An ordered pair $(T_1,T_2)$ of sets of integers are {\it separated} if, for $j=1,2$, 
$$\min(T_2)-\max(T_1) > \max(T_j)-\min(T_j).$$ 
\end{definition}

The next lemma shows that any dense subset $S$ contains a pair of large dense subsets which are separated. 

\begin{lemma}\label{sep}
Let $S$ be a finite set of integers with $|S| \geq 6$, and $I=[a,b]$ an interval with $S \subset I$. Then, for $j=1,2$, there is $T_j \subset S$ and an interval $I_j$ with $T_j \subset I_j$,  $(T_1,T_2)$ separated,  $d_{I_j}(T_j) \geq d_I(S)/2$, and $|I_j| \geq |S|/12$.  
\end{lemma}  
\begin{proof}
Let $i_1 \in I$ be the maximum integer (if it exists) such that the restriction of $S$ to the interval $[a,i_1]$ has 
density at most $d_I(S)/2$. If no such $i_1$ exists, let $i_1=a-1$. Similarly, let $i_2 \in I$ be the minimum integer greater than $i_1$ (if it exists) such that the restriction of $S$ to the interval $[i_2,b]$ has density at most $d_I(S)/2$. If no such $i_2$ exists, let $i_2=b+1$. Let $S'$ be the restriction of $S$ to the interval $(i_1,i_2)$, i.e., the set of $s \in S$ with $i_1<s<i_2$. 
Since at most $1/2$ of the elements of $S$ are deleted to obtain $S'$, we have $|S'| \geq |S|/2$.

Let $I'$ denote the interval $[\min(S'),\max(S')]$ of integers. Partition the interval $I'$ into three intervals each of size as equal as possible, and let $I_1$ be the first interval and $I_2$ be the last interval. This guarantees that if $T_1 \subset I_1$ and $T_2 \subset I_2$, then $(T_1,T_2)$ is separated. It follows from the definition of $i_1$ and $i_2$ that the restrictions of $S$ to each of the two end intervals has density at least $d_I(S)/2$. Let $T_j=|S \cap I_j|$ for $j=1,2$. Since $S' \subset I'$, we have $|I'| \geq |S'| \geq |S|/2$. The end intervals have size at least $\lfloor |I'|/3 \rfloor$. Hence, for $j=1,2$, $$|I_j| \geq \lfloor |I'|/3 \rfloor \geq \lfloor |S|/6 \rfloor \geq |S|/12.$$
The result follows.
\end{proof}

We also need the following simple lemma which allows us to pass to a subinterval of a given size without the density decreasing significantly.

\begin{lemma}\label{shrink}
Suppose $S$ is a set of positive integers, $J$ is an interval containing $S$, and $r \leq |J|$ is a positive integer. Then there is a subset $S' \subset S$ and an interval $I$ of size $r$ containing $S'$ such that $d_I(S') \geq  d_J(S)/2$. 
\end{lemma}
\begin{proof}
We can cover the interval $J$ with $\lceil |J|/r \rceil$ intervals of size $r$, some of which may be overlapping. If $S$ restricted to any of these intervals has density at least $d_J(S)/2$, then we can pick $S'$ to be this subset of $S$. Otherwise, since $\lceil |J|/r \rceil \leq 2|J|/r$, the total number of elements of $S$ is less than $$\lceil |J|/r\rceil rd_J(S)/2 \leq |S|,$$ a contradiction, which completes the proof. 
\end{proof}

For a permutation $\pi$ of $[k-1]$, an increasing sequence $a_1,\ldots,a_k$ of $k$ integers has {\it type} $\pi$ if 
$$a_{\pi(1)+1}-a_{\pi(1)}>a_{\pi(2)+1}-a_{\pi(2)}>\ldots>a_{\pi(k-1)+1}-a_{\pi(k-1)}.$$ 
Let $G$ be a graph on a subset of the integers, $J$ be an interval, and $S \subset J \cap V(G)$. For $0<\alpha,\beta,\gamma,\delta,p<1$, we say that $G$ is $(\alpha,\beta,\gamma,\delta,p)$-heavy with respect to $S$ if for all subsets $S' \subset S$ with $|S'| \geq \gamma|S|$ for which there is an interval $J'$ with $S' \subset J'$ and $d_{J'}(S') \geq \delta d_J(S)$ there are subsets $T_1, T_2 \subset S'$ and, for $j = 1, 2$, intervals $I_j$ with $T_j \subset I_j$ such that  $(T_1,T_2)$ is a separated pair, $d_{I_j}(T_j) \geq \alpha d_{J'}(S')$, $|I_j|\geq \beta |S'|$ and the edge density of $G$ across $T_1,T_2$ is at least $p$.  That is, $G$ is heavy with respect to $S$ if every large subset $S'$ of $S$ which is dense in an interval $J'$ contains a separated pair of subsets $T_1$ and $T_2$, dense in large intervals $I_1$ and $I_2$, such that there is a positive density of edges of $G$ between them.

Let $\phi:[h-1]\rightarrow [k-1]$ be an injective function, $0<\eta <1$, and $r \in \mathbb{N}$. A clique in $G$ of type $(\phi,\eta,r)$ consists of $h$ pairwise adjacent vertices $a_1,\ldots,a_h$ such that $a_{i+1}-a_{i} \in [\eta^{\phi(i)}r,\eta^{\phi(i)-1}r)$ for $i \in [h-1]$. Note that if $h=k$ and $\phi$ is the inverse permutation of $\pi$, then a clique of type  $(\phi,\eta,r)$ is also a clique of type $\pi$. 

The following lemma shows that if a large subset $S$ of a graph $G$ is $(\alpha,\beta,\gamma,\delta,p)$-heavy with appropriate choices of parameters $\alpha$, $\beta$, $\gamma$, $\delta$, and $p$, then it must contain a clique of type $(\phi,\eta,r)$. We next describe the proof, which is by induction on the order $h$ of the desired clique. The key idea is that we find the largest gap first. Let $\tau$ be the minimum element of the image of $\phi$, and $j$ be such that $\phi(j)=\tau$. We first pass to an interval $I$ of size just smaller than $\eta^{\tau-1}r$ 
using Lemma 
\ref{shrink}. Using the heavy hypothesis, we find a separated pair $(T_1,T_2)$ of large subsets of $S \cap I$ such that the edge density of $G$ between $T_1$ and $T_2$ is at least $p$, and $\min(T_2)-\max(T_1) \geq \eta^{\tau}r$. This implies that for any choice of $a_{j} \in T_1$ and $a_{j+1} \in T_2$, we have 
$a_{j+1}-a_j \in [\eta^{\tau}r,\eta^{\tau-1}r)$.
Applying the dependent random choice lemma, Lemma \ref{drc}, we find that there is a large subset $U \subset T_1$ such that all small subsets of $U$ have many common neighbors in $T_2$. We find from the heavy hypothesis and induction that there is a clique with vertices $a_1,\ldots,a_j \in U$ such that, for $1 \leq i \leq j-1$, $a_{i+1}-a_i \in [\eta^{\phi(i)}r,\eta^{\phi(i)-1}r)$. Since every small subset of $U$ has many common neighbors in $T_2$, the set $W$ of common neighbors of $a_1,\ldots,a_j$ in $T_2$ is large.  We again find from the heavy hypothesis and induction that there is a clique with vertices $a_{j+1},\ldots,a_h \in W$ such that, for $j+1 \leq i \leq h-1$, $a_{i+1}-a_i \in [\eta^{\phi(i)}r,\eta^{\phi(i)-1}r)$. We conclude that $a_1,\ldots,a_h$ forms the desired clique in $G$ of type $(\phi,\eta,r)$.

\begin{lemma}\label{keyfororder} 
Suppose $G$ is a graph on a subset of the integers, $J$ is an interval, $S \subset J \cap V(G)$, $\phi:[h-1] \rightarrow [k-1]$ is an injective 
function, $0 < \alpha, \beta, \gamma, \delta, \eta, p < 1$, and $r \in \mathbb{N}$. Let $t=2\sqrt{k\log_{1/p} |S|}$, $\epsilon=p^t/2$, $\lambda=\left(\frac{\epsilon 
\alpha}{4}\right)^{2h}$, and $\kappa= \lambda \beta d_J(S)^2\eta^kr$. Provided that $\kappa \geq h$, $|J| \geq r$, $\eta \leq \beta \lambda d_J(S)^2$, $\delta \leq \lambda$, and 
$\gamma|S| \leq \kappa$, the following holds. If $G$ is $(\alpha,\beta,\gamma,\delta,p)$-heavy with respect to $S$ then there is a clique in $G$ of type $(\phi,\eta,r)$. 
\end{lemma} 

\begin{proof} 
The proof is by induction on $h$. In the base case $h=1$, it suffices to show that $S$ is nonempty, which it clearly is. The induction hypothesis is that the lemma holds 
for all positive integers $h'<h$, where $h \geq 2$.

Let $\tau=\min_{i \in [h-1]} \phi(i)$ and $j \leq h-1$ be such that $\phi(j)=\tau$. Let $\phi_1:[j-1] \rightarrow [k-\tau-1]$ and $\phi_2:[h-j-1]\rightarrow [k-\tau-1]$ be the 
injective 
functions given by $\phi_1(x)=\phi(x)-\tau$ and $\phi_2(x)=\phi(x+j)-\tau$.

Let $s$ be the largest integer less than $\eta^{\tau-1}r$. Since $\eta^{\tau-1}r \geq \eta^{k}r \geq \kappa \geq h \geq 2$, then $s \geq \eta^{\tau-1}r/2$. As $|J| \geq r \geq 
s$, we can apply 
Lemma \ref{shrink} to obtain a subset $S' \subset S$ and an interval $I$ with $|I|=s$ and $S' \subset I$ such that $d_{I}(S') \geq d_J(S)/2$.

We have $d_I(S') \geq d_J(S)/2 \geq \lambda d_J(S) \geq \delta d_J(S)$ and 
$|S'|=d_I(S') |I| \geq \frac{d_J(S)}{2}|I| \geq \frac{d_J(S)}{4}\eta^{\tau-1}r \geq \kappa \geq \gamma |S|$. Hence, by 
the heaviness hypothesis, for $i=1,2$, there is an interval $I_i$ and a subset $T_i \subset I_i \cap S'$ such that $(T_1,T_2)$ is a separated pair, $d_{I_i}(T_i) \geq \alpha d_I(S') 
\geq \frac{\alpha}{2}d_J(S)$, $|I_i| \geq \beta |S'|$ and the edge density of $G$ between $T_1$ and $T_2$ is at least $p$.
Note that $|T_i| = |I_i| d_{I_i}(T_i) \geq d_{I_i}(T_i) \beta |S'| \geq \frac{1}{2} \alpha \beta d_J(S) |S'|$.
 
We apply Lemma \ref{drc} to the bipartite subgraph of $G$ with parts $T_1$ and $T_2$ and $s=j$, with $t$ as defined in the statement of the lemma, $N_1=|T_1|$, $N_2=|T_2|$, and 
$m=\epsilon|T_2|$. Since $|T_1|^k \leq |S|^k = p^{-\frac{1}{4}t^2}$, we can verify that $${|T_1| \choose j}\left(\frac{\epsilon|T_2|}{|T_2|}\right)^t \leq |T_1|^k \epsilon^t = |T_1|^k 
p^{t^2}/2^t \leq p^{\frac{3}{4}t^2}/2^t \leq p^t|T_1|/2.$$ 
Using that $|T_1|\geq \frac{1}{2} \alpha \beta d_J(S) |S'|$, $|S'| \geq \frac{d_J(S)}{4}\eta^{\tau-1}r $ and $\eta \leq \beta \lambda d_J(S)^2 \leq \beta 
\left(\frac{\epsilon \alpha}{16}\right)
d_J(S)^2$,
we conclude that there is a subset $U \subset T_1$ with $$|U| \geq p^{t}|T_1|/2 = \epsilon |T_1| \geq \epsilon\alpha\beta 
\frac{d_J(S)}{2} |S'|\geq \epsilon \alpha \beta \frac{d_J(S)^2}{8}\eta^{\tau-1}r \geq \eta^\tau r$$ such that every $j$ vertices in $U$ have at least $\epsilon |T_2|$ common neighbors 
in $T_2$. Since $(T_1,T_2)$ is separated and $|T_1| \geq |U| \geq \eta^\tau r$ we have that
for any $a \in T_1$ and $b \in T_2$, 
$$\eta^{\tau}r \leq |T_1| \leq  b-a \leq |I| < \eta^{\tau-1} r.$$

We also have $$d_{I_1}(U)=\frac{|U|}{|I_1|} \geq \frac{\epsilon|T_1|}{|I_1|} = \epsilon d_{I_1}(T_1) \geq \epsilon \alpha d_I(S') \geq \epsilon \frac{\alpha}{2}d_J(S).$$ 

Let $\delta' = \frac{d_J(S)}{d_{I_1}(U)}\delta$ and $\gamma'=\frac{|S|}{|U|}\gamma$. Since $G$ is $(\alpha,\beta,\gamma,\delta,p)$-heavy with respect to $S$ and $U \subset S$, then $G$ 
is also $(\alpha,\beta,\gamma',\delta',p)$-heavy with respect to $U$.

Let $t'=2\sqrt{(k-\tau)\log_{1/p} |U|}$, $k'=k-\tau$, $r'=\eta^{\tau}r$ and $\epsilon'=p^{t'}/2$, so $\epsilon' \geq \epsilon$. 
Let $\lambda'=\left(\frac{\epsilon' \alpha}{4}\right)^{2j}$ and $\kappa'= \lambda' \beta d_{I_1}(U)^2\eta^{k'} r'$. 
Then $\lambda \leq \left(\frac{\epsilon \alpha}{4}\right)^{2} \lambda'$ and therefore  
$$\kappa'= \lambda' \beta d_{I_1}(U)^2\eta^{k'}r' \geq \lambda' \beta \left(\epsilon \frac{\alpha}{2}d_J(S)\right)^2 \eta^kr \geq
\lambda \beta d_J(S)^2\eta^kr=\kappa \geq h \geq j.$$ 
Since $|I_1| \geq |U| \geq  \eta^{\tau}r=r'$, $\delta \leq \lambda=\left(\frac{\epsilon
\alpha}{4}\right)^{2h} \leq \left(\frac{\epsilon \alpha}{4}\right)^{2}\lambda'$, $\epsilon' \geq \epsilon$ and $d_{I_1}(U) \geq \frac{\epsilon \alpha}{2}d_J(S)$ we have that
 $$\eta \leq \beta d_J(S)^2 \lambda \leq \beta 
d_{I_1}(U)^2 
\lambda',$$ 
$$\delta'=\frac{d_J(S)}{d_{I_1}(U)}\delta \leq 2\epsilon^{-1}\alpha^{-1}\delta \leq \left(\frac{\epsilon\alpha}{4}\right)^{2h-1} \leq 
\lambda',$$ and $$\gamma'|U| = \gamma |S| \leq \kappa \leq \kappa' $$ 
Thus, we can apply the induction 
hypothesis and obtain  a clique in $G$  with vertices $a_1,\ldots,a_{j}$ in $U$ which is of type $(\phi_1,\eta,\eta^{\tau}r)$.  

Let $W$ be the set of common neighbors of $a_1,\ldots,a_j$ in $T_2$, so $|W| \geq \epsilon |T_2|$. Let $\delta'' = \frac{d_J(S)}{d_{I_2}(W)}\delta$ and $\gamma''=\frac{\gamma|S|}{|W|}$. As above, since $W \subset S$ and $G$ is $(\alpha,\beta,\gamma,\delta,p)$-heavy with respect to $S$, we have that $G$ is also $(\alpha,\beta,\gamma'',\delta'',p)$-heavy with respect to $W$. Again, by the induction hypothesis (exactly as done above, replacing $U$ by $W$ and $j$ by $h-j$), there is a clique $b_1,\ldots,b_{h-j}$ in $G$ with vertices from $W$ of type $(\phi_2,\eta,\eta^{\tau}r)$. Then, letting $a_{j+i}=b_i$ for $1 \leq i \leq h-j$, we have that $a_1,\ldots,a_h$ form a clique of type $(\phi,\eta,r)$ in $G$, completing the proof.  
\end{proof}

The following theorem is a restatement of Theorem \ref{shelahorder}. Recall that if $h=k$ and $\phi$ is the inverse permutation of $\pi$, then a clique of type  $(\phi,\eta,r)$ is also a clique of type $\pi$. In the proof of Theorem \ref{shelahorder}, we show that a $q$-colored complete graph on sufficiently many vertices must contain a subset which is appropriately heavy in the graph of one of the colors. Lemma \ref{keyfororder} then implies that the graph of this color contains the desired monochromatic clique with order type $\pi$. To find such a heavy subset, we suppose for contradiction that none exists. We then find a large interval $I_q$ and a dense subset $S_q$ of $I_q$ such that for each color $i$, every separated pair $(T_1,T_2)$ of subsets of $S_q$ and large intervals $J_1,J_2$ with $T_j$ a dense subset of $J_j$ has edge density less than $p=1/q$ in color $i$ between $T_1$ and $T_2$. But, by Lemma \ref{sep}, $S_q$ contains a separated pair $(T_1,T_2)$ of large dense subsets. By the pigeonhole principle, the edge density between $T_1$ and $T_2$ in one of the $q$ colors is at least $1/q$, contradicting the existence of $S_q$.

\begin{theorem}
Let $k,q \geq 2$ be integers and $\pi$ a permutation of $[k-1]$. Every $q$-coloring of the complete graph on $[n]$ with $n=2^{k^{20q}}$ contains a monochromatic clique of type $\pi$. 
\end{theorem}
\begin{proof}
Suppose for contradiction that there is a $q$-coloring of the edges of the complete graph on $[n]$ without a monochromatic copy of $K_k$ of type $\pi$. We label the $q$ colors $1,\ldots,q$.  Let $S_0=I_0=[n]$, so $d_{I_0}(S_0)=1$ and $|S_0|=|I_0|=n$. Let $\phi=\pi^{-1}$, $p=1/q$, $t=2\sqrt{k\log_{1/p} n}$, and $\epsilon=p^t/2$. 

For $q \geq i \geq 1$, we define $\alpha_i,\beta_i,\gamma_i,\delta_i,\eta_i$ recursively as follows, starting with $i=q$. We have $\alpha_q=1/2$, $\delta_i=\left(\frac{\epsilon\alpha_i}{4}\right)^{2k}$, and $\alpha_{i}=\delta_{i+1}\alpha_{i+1}$. Explicitly, $\delta_{q-i}=\left(\frac{\epsilon}{8}\right)^{2k(2k+1)^{i}}$, and for $i \geq 1$, $\alpha_{q-i}= \frac{1}{2}\left(\frac{\epsilon}{8}\right)^{(2k+1)^{i}-1}$. 
Let $\Delta_0=1$ and $\Delta_i=\delta_i\Delta_{i-1}^2$ for $1 \leq  i \leq q$. Let $\Delta=\Delta_q$. We have from the explicit formula for $\delta_{q-i}$ that 
$$\Delta=  \delta_q \Delta_{q-1}^2 = \delta_q \delta_{q-1}^2 \Delta_{q-2}^4 = \cdots =  \prod_{i = 0}^{q-1} \delta_{q-i}^{2^i} = \prod_{i=0}^{q-1} \left(\frac{\epsilon}{8}\right)^{2k(4k+2)^{i}}  \geq \left(\frac{\epsilon}{8}\right)^{(4k+2)^{q}} \geq \left(\frac{\epsilon}{8}\right)^{(k+2)^{2q}-2}.$$ 

Let $\beta_q=1/12$. For each $i$, let  $\eta_i=\beta_i\Delta$, and $\gamma_i=\eta_i^{k+1}$, and, if $i<q$,  $\beta_i=\gamma_{i+1}\beta_{i+1}$. Explicitly, $\beta_{q-i}=\frac{1}{12}\left(\frac{\Delta}{12}\right)^{(k+2)^i-1}$, $\eta_{q-i}=\left(\frac{\Delta}{12}\right)^{(k+2)^i}$, and $\gamma_{q-i}=\left(\frac{\Delta}{12}\right)^{(k+1)(k+2)^i}$. 
Finally, let $\Gamma_0=1$ and $\Gamma_i=\gamma_i\Gamma_{i-1}$ for $1 \leq i \leq q$. Let $\Gamma=\Gamma_q$. We have 
\begin{eqnarray*}
\Gamma&=&\prod_{i = 1}^{q} \gamma_i =\left(\frac{\Delta}{12}\right)^{(k+2)^q-1} \geq \left(\frac{1}{12} \left(\frac{\epsilon}{8}\right)^{(k+2)^{2q}-2}\right)^{(k+2)^q-1} \\
&\geq& 
\left(\left(\frac{\epsilon}{8}\right)^{(k+2)^{2q}}\right)^{(k+2)^q} = \left(\frac{\epsilon}{8}\right)^{(k+2)^{3q}} \geq \left(\frac{\epsilon}{8}\right)^{k^{6q}}.
\end{eqnarray*}

We will next define a sequence of subsets $S_0 \supset S_1 \supset \ldots \supset S_q$  and a sequence of intervals $I_0 \supset I_1 \supset \ldots \supset I_q$ such that for each $i$, $1 \leq i \leq q$, we have 
\begin{itemize} 
\item $S_i \subset I_i$, 
\item $d_{I_i}(S_{i}) \geq \delta_id_{I_{i-1}}(S_{i-1}) \geq \Delta_i$, 
\item $|S_i| \geq \gamma_i|S_{i-1}| \geq \Gamma_i n$, and 
\item there is no separated pair $(T_1,T_2)$ with $T_1,T_2 \subset S_i$ and intervals $J_1,J_2$ such that, for $j=1,2$,  $T_j \subset J_j$, $d_{J_j}(T_j) \geq \alpha_i d_{I_i}(S_i)$, $|J_j| \geq \beta_i |S_i|$, and the graph in color $i$ has edge density at least $p$ between $T_1$ and $T_2$. 
\end{itemize}

We next show how to pick $S_i$ and $I_i$ having already picked $S_{i-1}$ and $I_{i-1}$.  Since the graph in color $i$ does not contain a clique of type $\pi$, it also does not contain 
a clique of type 
$(\phi,\eta_i,r_i)$ with $r_i=|I_i|$. We now wish to apply Lemma \ref{keyfororder} with $S=S_{i-1}$ to conclude that the graph in color $i$ is not    
$(\alpha_i,\beta_i,\gamma_i,\delta_i,p)$-heavy with respect to $S_{i-1}$. To do this, we must verify the assumptions of the lemma. 

Let $\lambda_i = \left(\frac{\epsilon \alpha_i}{4}\right)^{2k}$ and $\kappa_i = \lambda_i \beta_i d_{I_{i-1}}(S_{i-1})^2 \eta_{i}^k r_{i}$. Note that $\delta_i = \lambda_i$ and
$$\eta_i = \beta_i \Delta \leq \beta_i \Delta_i = \beta_i \delta_i \Delta_{i-1}^2 \leq \beta_i \lambda_i d_{I_{i-1}}(S_{i-1})^2.$$ 
We also have
$$\gamma_i |S_{i-1}| = \eta_i^{k+1} |S_{i-1}| \leq \beta_i \lambda_i  d_{I_{i-1}}(S_{i-1})^2 \eta_{i}^{k} |S_{i-1}| \leq  \lambda_i  \beta_i d_{I_{i-1}}(S_{i-1})^2 \eta_{i}^{k} r_i= \kappa_i.$$
Finally, since $\gamma_i |S_{i-1}| \geq \Gamma_i n \geq \Gamma n$ and $n = 2^{k^{20q}}$, we have
$$\kappa_i \geq \Gamma n \geq n \left(\frac{\epsilon}{8}\right)^{k^{6q}} = n \left(\frac{q^{-2 \sqrt{k \log_q n}}}{16}\right)^{k^{6q}} \geq n \left(2^{-k^{12q}}\right)^{k^{6q}} = n 2^{-k^{18q}} \geq k.$$
Here we used that $q^{-2 \sqrt{k \log_q n}} \geq 2^{-2 k^{10q + 1} \sqrt{\log q}} \geq 2^{-k^{12q} + 4}$.

We may therefore apply Lemma \ref{keyfororder}. Hence, there is a subset $S_i \subset S_{i-1}$ and an interval $I_i \subset I_{i-1}$ satisfying the four desired properties itemized above.

However, by Lemma \ref{sep}, $S_q$ contains a separated pair $(T_1,T_2)$ and intervals $J_1,J_2$ such that, for $j=1,2$, $T_j \subset J_j$,   $d_{J_j}(T_j) \geq d_{I_q}(S_q)/2$, and $|J_j| \geq |S_q|/12$. By the pigeonhole principle, for some $i$, $1 \leq i \leq q$, the density across $T_1,T_2$ in color $i$ is at least $1/q=p$. But 
$$\frac{1}{2} d_{I_q}(S_q) = \alpha_q d_{I_q}(S_q) \geq \alpha_q \delta_q d_{I_{q-1}}(S_{q-1}) = \alpha_{q-1} d_{I_{q-1}}(S_{q-1}) \geq \alpha_{q-2} d_{I_{q-2}}(S_{q-2}) \geq \cdots \geq \alpha_id_{I_{i}}(S_i)$$ 
and, similarly, $|S_q|/12 \geq \beta_i|S_i|$, contradicting that $S_i$ contains no such separated pair. 
\end{proof}

\section{Further remarks} \label{sect6}

\subsection{Asymptotics of maximum weight monochromatic cliques} 

A well-known conjecture of Erd\H{o}s states that the limit $\lim_{n \to \infty}\frac{\log r(n)}{n}$ exists. If this limit exists, denote it by $c_0$. We will assume the conjecture that $c_0$ exists. The bounds of Erd\H{o}s and Erd\H{o}s-Szekeres on Ramsey numbers imply that $\frac{1}{2} \leq c_0 \leq 2$. 

Recall that the weight of a set $S$ of integers greater than one is the sum of $1/\log s$ over all $s \in S$, and  
$f(n)$ is the maximum real number for which any red-blue edge-coloring of $K_n$ contains a monochromatic clique of weight at least $f(n)$. Theorem \ref{erdosrodl} shows that $f(n)$ is within a constant factor of $\log \log \log n$.  We further conjecture the constant factor. 
\begin{conjecture}\label{conclude} We have $$f(n)=\left(c_0^{-2}+o(1)\right)\log \log \log n,$$ 
where $c_0=\lim_{n \to \infty} \frac{\log r(n)}{n}$. 
\end{conjecture}

The construction of R\"odl described in the introduction can easily be modified to obtain $$f(n) \leq (c_0^{-2}+o(1))\log \log \log n.$$  Indeed, let $a=1+\epsilon$ with $\epsilon \rightarrow 0$ slowly as $n \to \infty$ (picking $\epsilon=1/\log \log \log n$ will do). Cover $[2,n]$ by intervals, where the $i$th interval is $[2^{a^{i-1}},2^{a^{i}})$  and has largest element less than $n_i:=2^{a^i}$. The number of intervals is $d=\lceil \frac{1}{\log a}\log \log n \rceil =O(\epsilon^{-1}\log \log n)$. Note that the logarithm of any two numbers in the same interval is within a factor $a=1+\epsilon$ of each other. We red-blue edge-color the complete graph on each of these intervals so as to minimize the order of the largest monochromatic clique in the interval. Then the weight of any monochromatic clique in the $i$th interval is at most $(1/\log n_i)(c_0^{-1}+o(1))\log n_i=c_0^{-1}+o(1)$, where the $o(1)$ term goes to $0$ as $n_i$ increases. We color between intervals monochromatic so as to minimize the order of the largest monochromatic clique with vertices in distinct intervals. The order of this monochromatic clique with vertices in distinct intervals is $(c_0^{-1}+o(1))\log d=(c_0^{-1}+o(1))\log \log \log n$. Hence, $f(n) \leq (c_0^{-1}+o(1))(c_0^{-1}+o(1))\log \log \log n=(c_0^{-2}+o(1))\log \log \log n$. 

In the other direction, a simple modification of the proof of Theorem \ref{erdosrodl} with a careful analysis gives the lower bound $$f(n) \geq \left(\frac{1}{4}-o(1)\right)\log \log \log n,$$ which would be sharp if the exponential constant in the upper bound for diagonal Ramsey numbers is best possible, i.e., if $c_0=2$. We next give a rough sketch of how to achieve this. 

One first constructs $d=(\log \log n)^{1-o(1)}$ intervals $S_i$ of the form $[n_i,2n_i)$ with  $n_i=i(\log \log n)^{o(1)}+\frac{1}{2}\log \log n$, where the $o(1)$ term slowly goes to $0$ as $n$ tends to infinity. After going through the proof, we obtain in each $S_i$ a red clique $R_i$ and a blue clique $B_i$, such that for each $i<j$, the complete bipartite graph between $R_i \cup B_i$ and $R_j \cup B_j$ is monochromatic. The monochromatic cliques $R_i$ and $B_i$ are chosen to be the largest monochromatic cliques of each color in a particular subset $T_{i,d} \subset S_{i}$  with $|T_{i,d}|=|S_i|^{1-o(1)}$. By the Erd\H{o}s-Szekeres estimate, we have $|R_i| \geq (r_i - o(1)) \log n_i$ and $|B_i| \geq (b_i - o(1)) \log n_i$ where $b_i$ and $r_i$ (asymptotically) satisfy $(b_i+r_i)\log \frac{(b_i+r_i)}{r_i}-b_i\log \frac{b_i}{r_i}=1$.

Consider the induced red-blue edge-coloring of the complete graph with one vertex $v_i$ from each $R_i \cup B_i$.  Assign vertex $v_i$ red weight $r_i$ and blue weight $b_i$. An appropriate variant of Lemma \ref{WeightedRamsey}, the weighted version of Ramsey's theorem, tells us that there is a monochromatic clique $v_{i_1}, v_{i_2}, \dots, v_{i_s}$ of large weight. Assuming without loss of generality that this clique is red, the tailored variant of Lemma \ref{WeightedRamsey} then tells us that the red weight of the clique is asymptotically at least $\frac{1}{4}\log d=(\frac{1}{4}+o(1))\log \log \log n$. This is obtained when for each $i$, $b_i=r_i=\frac{1}{2}+o(1)$ and the clique has size $\frac{1}{2} \log d$.  Let $S$ be the union of the $R_{i_j}$ with $1 \leq j \leq s$. As, for each $i<j$, the complete bipartite graph between $R_i \cup B_i$ and $R_j \cup B_j$ is monochromatic red, the set $S$ forms a monochromatic clique of weight $$\sum_{j \in S} \frac{1}{\log j} \geq  \left(\frac{1}{4}+o(1)\right)\log \log \log n.$$ 

The proof sketched above uses an application of both Ramsey's theorem and its weighted variant, so that the asymptotics of the lower bound on $f(n)$ are dictated by the bounds in these theorems. We believe that the optimal bounds should always follow, as above, from the diagonal case, in which case Conjecture \ref{conclude} would follow. 

\subsection{Weighted cliques with alternative weight functions}

One question which arises naturally is whether we can also find cliques of large weight for other weight functions. Let $w(i)$ be a weight function defined on all 
positive integers $n \geq a$ and let $f(n,w)$ be the minimum over all $2$-colorings of $[a, n]$ of the maximum weight of a monochromatic clique. In particular, if 
$w_1(i)=1/\log i$ and $a = 2$, then $f(n,w_1)=f(n)$.

The next interesting case is when $w_2(i) = 1/\log i \log \log \log i$, since, for any function $u(i)$ which tends to infinity with $i$, Theorem \ref{erdosrodl} 
implies that $f(n, u') \rightarrow \infty$, where $u'(i) = u(i)/\log i \log \log \log i$.  We may show also that $f(n, w_2) \rightarrow \infty$.

\vspace{0.2cm}
\noindent
{\bf Sketch of the proof.}\,
Suppose that we are using the weight function $w_2$. We consider the intervals $I_j=[n_j,2n_j)$  for which $2n_j \leq n$ with 
$\log \log n_j=10j\log \log \log n$. The number $d$ of such intervals is $\log \log n/10\log \log \log n$. 
By applying the methods used in the proof of Theorem \ref{RodlMain}, we may find $d$ sets $T_1, T_2, \dots, T_d$, with $T_j \subset I_j$, the 
collection of edges between $T_i$ and $T_j$ is monochromatic for every $i \neq j$, and each $T_j$ is the union of a red clique of size roughly $r_j \log n_j$ and a 
blue clique of size $b_j \log n_j$. Here $r_j$ and $b_j$ are chosen to satisfy the balancing condition stipulated by Lemma \ref{ErdSzek}. 
Any vertex in $T_j$ will have weight about $1/\log n_j \log \log \log n_j$, the full contribution of the red clique is $\Omega(r_j/\log \log \log n_j)=\Omega(r_j/(\log j+\log  \log \log \log n))=\Omega(r_j/\log \max(j, \log \log \log n))$,  and the blue clique is $\Omega(b_j/\log \max(j, \log \log \log n))$.

We may now treat the $T_j$ as though they were vertices with two weights in a graph whose edges have been $2$-colored. For $j \geq \log \log \log n$, the red weight is $r_j/\log j$ and the blue weight is $b_j/\log j$. For smaller $j$, the red weight is $r_j/\log \log \log \log n$ and the blue weight is $b_j/\log \log \log \log n$. However, there are so few such smaller $j$ that we will be able to safely ignore such vertices. 
We would like to repeat the argument above with this new graph on $d$ vertices. To begin, we consider $e \approx
\log \log d / 10\log \log \log d$ intervals $S_1, \dots, S_e$ in $[d]$, each of the form $[d_i,2d_i)$ with $\log \log d_i=10i\log \log \log d$. For the rest of the argument we only consider vertices $j$ in one of these intervals, so that $j \geq d_1 \geq \log \log \log n$ and $j$ has red weight $r_j/\log j$ and blue weight $b_j/\log j$. We may assume that $r_j$ and $b_j$ are each less than $(\log j)^2$, as otherwise the vertex $j$, or rather the red or blue subset of $T_j$, would be a monochromatic clique of weight $\Omega(\log j)=\Omega(\log \log \log \log n)$. By Lemma \ref{ErdSzek}, this also implies that all $r_j$ and $b_j$ are at least $1/(16\log \log j)$. Therefore the ratio between any two of $r_j$ and any two of $b_j$ is at most  $16\log^2 j \log \log \log j \leq (\log j)^3$ and hence
we may split each $S_i$ into $h_i=6\log \log d_i$ subsets, so that the $r_j$ and $b_j$ are within a factor $2$ 
of each other within each piece. That is, we are decomposing the interval $S_i$ into $S_{i,1}, \dots, S_{i, h_i}$ so that within any $S_{i,\ell}$ all $r_j$ and $b_j$ are essentially the same. Within each $S_i$, we pass to the largest $S_{i, \ell}$, which we will call $U_i$. As $|U_i| \geq d_i/(6\log \log d_i)$, we have $\log |U_i| \approx \log |S_i|$ for each $i$. We let $r_i'$ and $b_i'$ be  the minimum over $j \in U_i$ of $r_j$ and $b_j$, respectively.

If we again apply the method of Theorem \ref{RodlMain}, we will find a collection of sets $T'_i\subset S_i$ such that the graph is monochromatic between any two sets and $T'_i$ contains a red clique of size $\hat r_i \log |U_i| \approx \hat r_i \log 
|S_i|$ and a blue clique of size roughly $\hat b_i \log |S_i|$. The red clique will have red weight $\Omega(\hat  r_i r'_i)$ and the blue clique will have blue weight 
$\Omega(\hat b_i b'_i)$. Treating the $T'_i$ as though they were the vertices in a graph, we see that the vertex $i$ will have red weight $\Omega(\hat r_i r'_i)$ and blue 
weight $\Omega(\hat b_i b'_i)$, where $\hat r_i$ and $\hat b_i$ as well as $r'_i$ and $b'_i$ satisfy, up to a constant factor, the balancing criterion stipulated by Lemma 
\ref{ErdSzek}. It is now easy to verify that the weight functions $\hat r_i r'_i$ and $\hat b_i b'_i$ satisfy the requirements of Lemma 
\ref{WeightedRamsey2} with $c>0$ an appropriately chosen absolute constant. Hence, we will be able to find a monochromatic clique of weight $\Omega(\log e) = \Omega(\log \log \log d) = \Omega(\log \log \log \log \log n)$. This yields a clique of the same weight in the original graph. \hfill $\Box$

\vspace{0.2cm}

It is not hard to show that this bound is tight up to the constant. Color the interval $I_j = [2^{2^{j-1}}, 2^{2^j})$ so that the largest clique has size at most 
$2^{j+1}$.  Then the contribution of the $j$th interval will be at most $4/\log j$. We now treat $I_j$ as though it were a vertex of weight $4/\log j$ and, 
blowing up R\"odl's coloring, color monochromatically between the different $I_j$ so that the largest weight of any monochromatic clique is $O(\log \log \log d) = 
O(\log \log \log \log \log n)$.

On the other hand, by using R\"odl's coloring, we can show that if $w'_1 (i) = 1/(\log i)^{1+\epsilon}$, for any fixed $\epsilon > 0$, then $f(n, w'_1)$ 
converges. By using the coloring from the previous paragraph, we may improve this to show that if $w'_2 (i) = 1/\log i (\log \log \log i)^{1 + \epsilon}$, then 
$f(n, w'_2)$ also converges.

More generally, we have the following theorem. Here $\log_{(i)} (x)$ is the iterated logarithm given by $\log_{(0)}(x) = x$ and, for $i \geq 1$, $\log_{(i)} (x) = 
\log (\log_{(i-1)} (x))$.

\begin{theorem}
Let $w_s(i)=1/\prod_{j=1}^s \log_{(2j-1)} i$. Then $f(n,w_s) = \Theta(\log_{(2s+1)} n)$. However, letting
$w'_s(x)=w_s(x)/(\log_{(2s-1)} i)^{\epsilon}$ for any fixed $\epsilon > 0$, then $f(n,w'_s)$ converges.
\end{theorem}

That is, the sequence of functions $w_s$ form a natural boundary below which $f(n, \cdot)$ converges.

\subsection{A counterexample to finding skewed cliques in hypergraphs}

For $3$-uniform hypergraphs, the Ramsey number $r_3(t)$ is defined to be the smallest natural number $n$ such that in any $2$-coloring of the edges of $K_n^{(3)}$ there is a monochromatic copy of $K_t^{(3)}$. It is known (see \cite{CFS10, EHR65, ER52}) that
\[2^{c t^2} \leq r_3(t) \leq 2^{2^{c' t}}\]
and the upper bound is widely conjectured to be correct. Phrased differently, we know that every $2$-coloring of the edges of $K_n^{(3)}$ contains a monochromatic clique of size at least $\Omega(\log \log n)$ and that there are $2$-colorings of $K_n^{(3)}$ which contain no monochromatic clique of size $O(\sqrt{\log n})$. 

Let $\rho_3 (n)$ be the function which gives the minimum size of the largest monochromatic clique taken over every $2$-coloring of $K_n^{(3)}$. Note that this function is increasing and that $\rho_3(r_3(t)) = t$. In keeping with Erd\H{o}s' conjecture for graphs, we can give a weight of $1/\rho_3 (i)$ to vertex $i$ and let the weight of a set $S$ be $\sum_{i \in S} 1/\rho_3 (i)$. We then ask for the minimum over all $2$-colorings of the edges of the complete $3$-uniform hypergraph on vertex set $[n]$ of the maximum weight of a monochromatic clique.

Split $[n]$ into intervals given by $R_j = [r_3(2^{j-1}), r_3(2^j))$. Within each interval, we color so that the largest monochromatic clique has size at most $2^j$. If $i < j$, we color edges containing two vertices from $R_i$ and one vertex from $R_j$ red and edges containing two vertices from $R_j$ and one vertex from $R_i$ blue. We color all other edges arbitrarily. 

Suppose now that we have a monochromatic clique $S$. Then $S$ has at most one vertex in all but one of the sets $R_j$. Otherwise, if there were two vertices, say $u_1$ and $u_2$, in $R_i$ and two vertices, $v_1$ and $v_2$, in $R_j$, the edges $u_1 u_2 v_1$ and $u_1 v_1 v_2$ would have opposite color. We may therefore suppose that $S = T_{\ell} \cup \{s_1, s_2, \dots\}$, where $T_{\ell} \subset R_{\ell}$ and $s_i$ is a single vertex from $R_i$. 

Since, for any $i \in R_{\ell}$, we have $\rho(i) \geq \rho_3(r_3(2^{\ell-1})) = 2^{\ell-1}$ and the largest monochromatic clique in $R_{\ell}$ has size at most $2^{\ell}$, the contribution from $T_{\ell}$ is at most $2$. Similarly, the contribution from $s_i$ is at most $2^{1-i}$, so that total weight of the clique is at most $2 + \sum_{i=1}^{\infty} 2^{1-i} \leq 4$. Therefore, unlike the graph case, there are colorings for which the maximum weight of a monochromatic clique is bounded.

\subsection{A simple construction}\label{sect5}

Here we present a simple explicit construction which beats the random lower bound for Ramsey numbers for a certain prescribed order on the consecutive differences. A sequence 
$n_1<n_2<\ldots<n_k$ is {\it convex} if $n_2-n_1<n_3-n_2<\ldots<n_k-n_{k-1}$. 

\begin{proposition} \label{easyconv} For $i<j$, let $f(i,j)=\lfloor \log (j-i) \rfloor$. Consider the $2$-edge-coloring of the complete graph on the first $n=4^{k-1}$ positive integers where the color of edge $(i,j)$ with $i<j$ is the parity of $f(i,j)$. This coloring has no convex monochromatic clique of order $k+1$.
\end{proposition}
\begin{proof}
Suppose for contradiction that $a_1<\ldots<a_{k+1}$ is a convex monochromatic clique of order $k+1$ in this $2$-edge-coloring of the complete graph on $n$. We claim that for $1 \leq i \leq k-1$, $f(a_{i+2},a_{i+1}) \geq f(a_{i+1},a_i)+2$. Indeed, as the sequence is convex, $a_{i+1}-a_i<a_{i+2}-a_{i+1}$, and hence $f(a_{i+2},a_{i+1}) \geq f(a_{i+1},a_i)$. If the claim does not hold, then  
for some $i$, $1 \leq i \leq k-1$, we have $f(a_{i+2},a_{i+1}) = f(a_{i+1},a_i)$ or $f(a_{i+2},a_{i+1}) = f(a_{i+1},a_i)+1$. In the first case, as $a_{i+2}-a_{i}=(a_{i+2}-a_{i+1})+(a_{i+1}-a_i)$, we have that $
f(a_{i+2},a_i)=f(a_{i+1},a_i)+1$, so the edges $(a_i,a_{i+2})$ and $(a_i,a_{i+1})$ are different colors. In the second case, $(a_{i+2},a_{i+1})$ and $(a_{i+1},a_i)$ are different colors. As the clique is monochromatic, this cannot happen, and hence the claim holds. From the claim, we have $f(a_{k+1},a_k) \geq f(2,1)+2(k-1) \geq 2(k-1)$. It follows that $a_{k+1} > a_{k+1}-a_k \geq 2^{2(k-1)}$, contradicting $a_{k+1} \leq n=4^{k-1}$ and completing the proof. 
\end{proof}

We actually proved that not only is there no convex monochromatic complete graph on $k+1$ vertices in the $2$-edge-coloring of the complete graph on the first $4^{k-1}$ positive 
integers, but also a much sparser graph on $k+1$ vertices is forbidden as a monochromatic subgraph in convex position, namely, the square of the monotone path on $k+1$ vertices. That 
is, for this coloring, there is no convex sequence $a_1,\ldots,a_{k+1}$ such that all edges $(a_i,a_j)$ with $|j-i| \leq 2$ are the same color. This is in strong contrast to Ramsey 
numbers without order, where the Ramsey number of the square of a path or, more generally, any bounded degree graph (see, e.g., \cite{CRST, CFS11}) is linear in the 
number of vertices.

As with ordinary Ramsey numbers, the lower bound for complete Ramsey numbers with order types which comes from considering a random $2$-edge-coloring of the complete graph is of the 
form $2^{k/2+o(k)}$. As the simple constructive coloring in Proposition \ref{easyconv} gives a better bound while forbidding a much sparser structure, it suggests that Ramsey's theorem 
with order types is a substantially different and more intricate problem than Ramsey's theorem.

\subsection{Counterexamples to variants of Ramsey's theorem with order types} 

There are several natural variants of V\"a\"an\"anen's question which have negative answers. For example, the  natural hypergraph analogue fails. Indeed, there is a coloring of the complete $3$-uniform hypergraph on the positive integers such that every  
monochromatic set $a_1,\ldots,a_k$ satisfies that the sequence $a_2-a_1,a_3-a_2,\ldots,a_k-a_{k-1}$ of consecutive differences is monotone. We color an edge $(a_1,a_2, a_3)$ with $a_1<a_2<a_3$ red if $a_3-a_2 \geq a_2-a_1$ and blue otherwise. Hence, if  $a_1<a_2<a_3<a_4$ are positive integers, $(a_1,a_2,a_3)$ and $(a_2,a_3,a_4)$ are both red or both blue if and only if $a_2-a_1,a_3-a_2,a_4-a_3$ is a monotone sequence. 

Another variant which fails to hold is the case of monochromatic cliques where the higher differences have a prescribed order. This was first observed by Erd\H{o}s, Hajnal and Pach 
\cite{EHP}. We give such an example forbidding an ordering of the second differences $a_{i+2} - a_i$. Before describing this coloring, we first remark that it is easy to show that 
any second difference is realizable. That is, for any permutation $\pi$ of $[k-2]$, there are (many) sequences $a_1<\cdots<a_k$ of positive integers satisfying
$$a_{\pi(1)+2}-a_{\pi(1)}>a_{\pi(2)+2}-a_{\pi(2)}>\cdots>a_{\pi(k-2)+2}-a_{\pi(k-2)}.$$ 
However, for certain $\pi$ there exist $2$-edge-colorings of the complete graph on the positive 
integers in which none of these sequences form a monochromatic clique. Indeed, consider the $2$-edge-coloring of the complete graph on the positive integers, where the color of $(i,j)$ 
with $i<j$ is given by the parity of $f(i,j)=\lfloor \log (j-i) \rfloor$. In this coloring, no monochromatic clique with vertices $a_1<a_2<a_3<a_4<a_5<a_6<a_7$ satisfies $a_5-a_3$ is 
the largest of the second differences and $a_4-a_2,a_6-a_4$ are the two smallest second differences. Suppose that such a monochromatic clique exists. By symmetry, we may assume without loss of generality that $a_4-a_3 \geq a_5-a_4$. For $a_i<a_j<a_h$, as $a_h-a_i=(a_h-a_j)+(a_j-a_i)$, we have $\max(f(a_i,a_j),f(a_j,a_h))\leq f(a_i,a_h)\leq \max(f(a_i,a_j),f(a_j,a_h))+1$. 
Since the parity of $f(a,b)$ is the same for any two vertices $a<b$ of the monochromatic clique, we must have $f(a_i,a_h)=\max(f(a_i,a_j),f(a_j,a_h))$. In particular, this implies 
$f(a_3,a_5)=f(a_3,a_4)$ and $f(a_1,a_5)=f(a_3,a_5)$. Since $a_3-a_1 \geq a_4-a_2$ (by minimality of $a_4-a_2$), we must have $a_2-a_1 \geq a_4-a_3$ and hence $f(a_3,a_5) \geq 
f(a_1,a_3) \geq f(a_1,a_2) \geq 
f(a_3,a_4)=f(a_3,a_5)$, where the first inequality comes from the fact that $a_5-a_3$ is the largest second difference. But if $f(a_1,a_3)=f(a_3,a_5)$, then $f(a_1,a_5)>f(a_3,a_5)$, 
contradicting the equality deduced earlier.

{\bf Acknowledgments.} We would like to thank Noga Alon for helpful discussions and, in particular, for raising the question of what other weight functions might work in Erd\H{o}s' conjecture. We would also like to thank the anonymous referees for their many helpful remarks.

\end{document}